\documentclass[11pt]{amsart}
\usepackage{latexsym,amsfonts,amssymb,amsthm,amsmath,mathrsfs,color,amscd,graphicx,fullpage,parskip,hyperref,bm,bbm,dsfont}

\usepackage{comment}
\includecomment{uncomment}
\usepackage{commath,mathtools,setspace}
\usepackage{paralist}
\usepackage{extarrows}
\usepackage{array}
\newcolumntype{M}[1]{>{\centering\arraybackslash}m{#1}}

\excludecomment{comment} 

\newcommand{\s}[1]{{\mathcal #1}}
\newcommand{\sr}[1]{{\mathscr #1}}
\newcommand{\bb}[1]{{\mathbb #1}}

\newcommand{\firststep}{\setcounter{step}{1}\textbf{Step \arabic{step}:} }
\newcommand{\nextstep}{\stepcounter{step}\textbf{Step \arabic{step}:} }

\DeclareMathOperator{\Lip}{Lip}

\newtheorem{theorem}{Theorem} 
\newtheorem{corollary}[theorem]{Corollary}

\newtheorem{lemma}[theorem]{Lemma}
\newtheorem{proposition}[theorem]{Proposition}

\newtheorem{definition}[theorem]{Definition}

\newtheorem{remark}[theorem]{Remark}
\newtheorem{assumption}[theorem]{Assumption}

\numberwithin{equation}{section}
\numberwithin{theorem}{section}

\newcounter{step}
\begin{document}

	\title[MFG Iterated]
	{Cluster formation in iterated Mean Field Games}
	
	\author{P. Jameson Graber}
	\thanks{Research supported by the National Science Foundation under NSF Grant DMS-2045027.}
	\address{J. Graber: Baylor University, Department of Mathematics;\\
		Sid Richardson Building\\
		1410 S.~4th Street\\
		Waco, TX 76706\\
		Tel.: +1-254-710- \\
		Fax: +1-254-710-3569 
	}
	\email{Jameson\_Graber@baylor.edu}
	
	\author{Ellie Matter}
	\address{E.~Matter: Baylor University, Department of Mathematics;\\
		Sid Richardson Building\\
		1410 S.~4th Street\\
		Waco, TX 76706
	}
	\email{Ellie\_Carr3@baylor.edu}
	
	\author{Rafael Morales}
	\address{R.~Morales: Baylor University, Department of Mathematics;\\
		Sid Richardson Building\\
		1410 S.~4th Street\\
		Waco, TX 76706
	}
	\email{Rafael\_Morales2@baylor.edu}
	
	\author{Lindsay North}
	\address{L.~North: Baylor University, Department of Mathematics;\\
		Sid Richardson Building\\
		1410 S.~4th Street\\
		Waco, TX 76706
	}
	\email{Lindsay\_North1@baylor.edu}
	
	\date{\today}   
	
	\begin{abstract}
		We study a simple first-order mean field game in which the coupling with the mean field is only in the final time and gives an incentive for players to congregate.
		For a short enough time horizon, the equilibrium is unique.
		We consider the process of \emph{iterating} the game, taking the final population distribution as the initial distribution in the next iteration.
		Restricting to one dimension, we take this to be a model of coalition building for a population distributed over some spectrum of opinions.
		Our main result states that, given a final coupling of the form $G(x,m) = \int \varphi(x-z)\dif m(z)$ where $\varphi$ is a smooth, even, non-positive function of compact support, then as the number of iterations goes to infinity the population tends to cluster into discrete groups, which are spread out as a function of the size of the support of $\varphi$.
		We discuss the potential implications of this result for real-world opinion dynamics and political systems.
	\end{abstract}
	
	\keywords{mean field games}
	
	\maketitle
	
	
	\section{Introduction}
	
	Mean field games describe the strategic Nash equilibrium behavior of large (continuum) populations \cite{lasry07,cardaliaguet2010notes,carmona2017probabilistic,carmona2017probabilisticII,bensoussan2013mean}.
	A typical mean field game of first order can be described by a system of PDE as follows:
	\begin{equation}
		\begin{split}
			-\partial_t u + H(x,\nabla_x u) &= F(x,m),\\
			\partial_t m - \nabla_x \cdot \del{D_p H(x,\nabla_x u)m} &= 0,\\
			m(x,0) = m_0(x), \quad u(x,T) &= G(x,m(x,T)).
		\end{split}
	\end{equation}
	The \emph{Lasry-Lions monotonicity condition} \cite{lasry07} can be defined as
	\begin{equation}
		\int \del{F(x,m_1) - F(x,m_2)}\dif (m_1 - m_2)(x) \geq 0 \quad \forall m_1,m_2 \in \sr{P}(\bb{R}^d).
	\end{equation}
	It is well-known that when $F$ and $G$ satisfy the Lasry-Lions monotonicity condition, players have an incentive to spread out from one another, and for this reason one can expect uniqueness of the equilibrium on an arbitrarily long time horizon $T$.
	If we assume, to the contrary, that $F$ and $G$ exhibit an \emph{anti-}monotonicity property, then players have an incentive to \emph{congregate}, i.e.~form clusters.
	It is this latter case which interests us in the present study.
	In such a case, we do not in general expect uniqueness of the equilibrium, unless the time horizon $T$ is sufficiently small.
	Cf.~\cite{bardi2019non,bayraktar2020non,cirant2016stationary,cirant2019time} for some general results on existence, non-existence, and non-uniqueness under this anti-monotonicity assumption.
	If we assume a sufficiently small time horizon so as to guarantee the problem is well-posed, then even if we do see some clustering behavior, we do not expect it to be significant after playing the game only over a short time horizon. 
	A natural question is, what happens if we iterate the game?
	That is, suppose we take the final distribution of players to be our new \emph{initial} distribution, and then repeat this process over many iterations?
	Such an iterative process can be taken as a model for population dynamics resulting from a string of small strategic decisions on the part of individuals.
	We expect that after sufficiently many iterations, we will see players accumulate in tighter and tighter clusters.
	The purpose of this study is to see whether this in fact occurs and, if it does, to locate the clusters that emerge.
	
	We are going to keep the focus on the clustering phenomenon described above and not on abstract results for mean field games.
	For this reason we posit a simple game by making the following assumptions.
	We will take $F = 0$, so that only the final cost depends on the distribution.
	In addition, we will assume $H(x,p) = \frac{1}{2}\abs{p}^2$.
	Then the game becomes equivalent to the following fixed point problem.
	Define
	\begin{equation}
		J_t(x,y,m) = \frac{\abs{x-y}^2}{2t} + G(x,m).
	\end{equation}
	Let $y_t^m(x) := \operatorname{argmin} J_t(x,\cdot,m)$.
	We will specify below conditions under which $y_t^m(x)$ is a well-defined (single-valued) function.
	Define $F_{m_0}:\sr{P}(\bb{R}^d) \to \sr{P}(\bb{R}^d)$ by
	\begin{equation}
		F_{m_0}(m) = y_t^m \sharp m_0,
	\end{equation}
	where $m_0 \in \sr{P}(\bb{R}^d)$ is the given initial measure.
	Then $m$ is a Nash equilibrium if and only if $m = F_{m_0}(m)$.
	The interpretation of this game is that individuals are willing to make a small move from their present state if and only if it will land them in a region of sufficiently high population density to justify the move.
	
	By taking $t$ sufficiently small, we will be able to ensure that $F$ has a unique fixed point in the Wasserstein space $\sr{P}_1(\bb{R}^d)$, defined below.
	We now define the \emph{equilibrium map} $E:\sr{P}_1(\bb{R}^d) \to \sr{P}_1(\bb{R}^d)$ to be the map which, for any given initial measure $m_0$, outputs the equilibrium measure $m = F_{m_0}(m)$.
	We then wish to consider the following dynamical system:
	\begin{equation} \label{eq:dynamical system}
		\begin{split}
			m_0 &\quad \text{is given},\\
			m_{k+1} &= E(m_k), \quad k = 0,1,2,\ldots
		\end{split}
	\end{equation}
	The main theoretical contributions of this paper are as follows:
	\begin{enumerate}
		\item We give a straightforward algorithm to reliably compute $E(m)$ by \emph{discretizing the measure}, i.e.~by putting an empirical measure in its place.
		\item We study the asymptotic behavior of the dynamical system \eqref{eq:dynamical system}.
		We prove that for dimension $d = 1$ and couplings of the form $G(x,m) = \int \varphi(x-y)\dif m(y)$, where $\phi$ is an even, non-positive, smooth function of compact support, we can explicitly locate the fixed points of $E$ and show that they are asymptotically stable.
	\end{enumerate}
	The motivation behind this analysis is an interest in the implications for human behavior.
	Humans instinctually like to form groups. People form groups of different sizes for numerous reasons, but specifically, people join social groups to access grouping utility that is otherwise unavailable to lone individuals. Social groups also benefit when new members join. A larger platform could provide the necessary leverage to accomplish a group's goals. Moreover, group membership transforms individuals by providing them with a group identity and inviting members to take part in a shared belief system, and individuals affect their groups by adopting organizational roles and relating to other members. How people group and their groups' actions can completely change a population's status quo and reshape society at large. Cf.~\cite{coleman1994foundations,converse1964nature,forsyth2014group,van2013ideology}. Given the impact of grouping behavior on society, how does a population arrange itself considering only the individual proclivity to group? 
	
	The results of this paper show that a population may ``cluster'' into multiple discrete groups, despite having no aversion whatsoever to one another.
	Specifically, assume that $G(x,m) = \int \varphi(x-y)\dif m(y)$ is as in point (2) above and that the initial population distribution is sufficiently spread out over an interval compared to the radius of the support of $\varphi$.
	Then after a sufficient number of iterations of the game, players will myopically cluster into small clusters sufficiently separated from each other so that they will not see any incentive to move toward each other.
	For example, consider a population whose ideological positions on an issue are distributed over some interval (say, between ``left'' and ``right''). 
	If the initial distribution is sufficiently spread out, then over time the incentive to congregate will in fact induce ``polarization,'' i.e.~the formation of two distinct groups with no incentive to move toward one another.
	This occurs despite the fact that no individual has any preference for moving one direction over another; they are driven solely by the incentive to find themselves in a crowded region.
	Somewhat counter-intuitively, the desire to maximize crowd size results in multiple insular groups rather than all players congregating in the middle.
	We find this to be a thought-provoking result that may have implications for social science. 
	
	In the remainder of this introduction, we introduce some notation and assumptions on the data which will hold throughout this study.
	Then in Section \ref{sec:computation} we provide a simple proof that the problem is well-posed and give an algorithm to explicitly compute solutions when the initial condition is a given discrete (empirical) measure.
	In Section \ref{sec:asymptotic} we study the dynamical system \eqref{eq:dynamical system} over the set of discrete measures; this section contains our key theoretical results.
	Section \ref{sec:numerics} discusses numerical simulations that illustrate our main results.
	Finally, we give some concluding remarks in Section \ref{sec:conclusion}.
	
	
	\subsection{Notation and assumptions}
	
	We denote by $\s{P}_1(\bb{R}^d)$ the set of all Borel probability measures $m$ on $\bb{R}^d$ such that the first moment $\int \abs{x}\dif m(x)$ is finite.
	It will be endowed with the Wasserstein metric
	\begin{equation}
		W_1(\mu,\nu) = \inf\cbr{\int \abs{x-y}\dif \pi(x,y) : \pi \in \Pi(\mu,\nu)}
	\end{equation}
	where $\Pi(\mu,\nu)$ denotes the set of all couplings between $\mu$ and $\nu$.
	By Kantorovitch duality, we also have the characterization
	\begin{equation}
		W_1(\mu,\nu) = \sup\cbr{\int \phi \dif(\mu-\nu) : \enVert{\nabla \phi}_\infty \leq 1}.
	\end{equation}
	When $A$ and $B$ are symmetric matrices, we write $A \geq B$ to mean that $A-B$ is positive semi-definite, i.e.~all of its eigenvalues are non-negative.

	The basic assumptions on the data are given here:
	\begin{assumption}\label{Gassum}
		$G:\bb{R}^d \times \s{P}_1(\bb{R}^d) \to \bb{R}$ is continuous in both variables, twice differentiable with respect to the variable $x \in \bb{R}^d$.
		$D_x G(x,m)$ is $L_1$-Lipschitz with respect to the measure variable in the $W_1$ metric, i.e.~
		\begin{equation}
			\abs{D_x G(x,m) - D_x G(x,\tilde m)} \leq L_1 W_1(m,\tilde{m}) \quad \forall x \in \bb{R}^d, \ \forall m,\tilde{m} \in \s{P}_1(\bb{R}^d).
		\end{equation}
		Both $D_x G$ and $D_{xx}^2 G$ are bounded. 
		In particular, there exists $\lambda_1(G) \geq 0$ such that $D^2_{xx}G(x,m) + \lambda_1(G)I \geq 0$ for all $x,m$.
	\end{assumption}
	
	\section{Computation of solutions} \label{sec:computation}
	
	The purpose of this section is to propose an algorithm to compute $E(m)$, based on discretization of the measure.
	To begin with, however, we prove that our problem is indeed well-posed.
	We then provide a simple algorithm for computing solutions when $m$ is an empirical measure, i.e.~a sum of Dirac masses.
	
	\subsection{Well-posedness} \label{sec:well-posed}
	\begin{theorem} \label{thm:well-posed}
		Assume $0 < t^* < \frac{1}{\lambda_1(G) + L_1}$, where $L_1$ and $\lambda_1(G)$ are defined in Assumption \ref{Gassum}.
		If $0 < t \leq t^*$, then
		\begin{enumerate}
			\item $J_t(x,\cdot,m)$ has a unique minimizer for every $x$ and $m$, hence $y_t^m(x)$ and $F_{m_0}(m)$ are well-defined;
			\item $F_{m_0}$ is a contraction on $\s{P}_1(\bb{R}^d)$ and therefore has a unique fixed point, so that $E$ is well-defined;
			\item $E$ is continuous with the respect to the Wasserstein metric, i.e.~if $\{m_{0,n}\}$ is a sequence in $\s{P}_1(\bb{R}^d)$ such that $W_1(m_{0,n},m_{0}) \to 0$, then $W_1\del{E(m_{0,n}), E(m_0)} \to 0$.
		\end{enumerate}
	\end{theorem}
	
	\begin{proof}
		\firststep

		To find the minimizer for $J_t(x,\cdot, m)$ we start by taking the derivative with respect to $y$ and setting it equal to $0$
		\begin{equation}
			D_y[J_t(x,y,m)]=\frac{y-x}{t}+D_yG(y,m)=0
		\end{equation}
		So we see there is a critical point: 
		\begin{equation}\label{eq:critpt}
			y+tD_yG(y,m)=x 
		\end{equation}
		Then taking the 2nd derivative with respect to $y$ and using \ref{Gassum}, we see
		\begin{equation}
			D^2_{yy}[J_t(x,y,m)]=\frac{1}{t}I+D^2_{yy}G(y,m) \geq \del{\frac{1}{t}-\lambda_1(G)}I
		\end{equation}
		where $I$ is the identity matrix.
		Thus $D^2_{yy}[J_t(x,y,m)]$ is positive definite provided that $0<t\leq t^* < \frac{1}{\lambda_1(G)}$, and since $\frac{1}{\lambda_1(G) + L_1} \leq \frac{1}{\lambda_1(G)}$ this is satisfied.
		Then by the second derivative test, the critical point in \ref{eq:critpt} is the unique minimizer.

		\nextstep

		Let $\varphi$ be Lipschitz with $\enVert{\varphi}_{\Lip}\leq1$
		\begin{equation}
				\int\varphi \dif(y_t^m\#m_0-y_t^{\tilde{m}}\#m_0)=\int[\varphi(y_t^m(x))-\varphi(y_t^{\tilde{m}}(x))]\dif m_0(x)
				\leq \int |y_t^m(x)-y_t^{\tilde{m}}(x)|\dif m_0(x).
		\end{equation}
		Using \eqref{eq:critpt}, we get
		\begin{equation}
			\begin{split}
				y_t^m-y_t^{\tilde{m}}&=x-tD_yG(y_t^m,m)-x+tD_yG(y_t^{\tilde{m}},\tilde{m})\\
				&=-t[D_yG(y_t^m,m)-D_yG(y_t^m,\tilde{m})+D_yG(y_t^m,\tilde{m})-D_yG(y_t^{\tilde{m}},\tilde{m})]
			\end{split}
		\end{equation}
		Taking the dot product of each side with $y_t^m-y_t^{\tilde{m}}$ yields 
		\begin{equation}
			\begin{split}
				|y_t^m-y_t^{\tilde{m}}|^2=-t[(y_t^m-y_t^{\tilde{m}})\cdot(D_yG(y_t^m,m)-D_yG(y_t^m,\tilde{m}))\\
				+(y_t^m-y_t^{\tilde{m}})\cdot(D_yG(y_t^m,\tilde{m})-D_yG(y_t^{\tilde{m}},\tilde{m}))]
			\end{split}
		\end{equation}
		Since $(y_t^m-y_t^{\tilde{m}})\cdot(D_yG(y_t^m,\tilde{m})-D_yG(y_t^{\tilde{m}},\tilde{m}))\geq -\lambda_1(G)\abs{y_t^m - y_t^{\tilde{m}}}^2$,
		\begin{equation}
			(1-t\lambda_1(G))|y_t^m-y_t^{\tilde{m}}|^2\leq |t(y_t^m-y_t^{\tilde{m}})\cdot(D_yG(y_t^m,m)-D_yG(y_t^m,\tilde{m}))|
		\end{equation}
		Now we have
		\begin{equation} \label{eq:ym-ym'}
				|y_t^m-y_t^{\tilde{m}}|\leq \frac{t}{1-t\lambda_1(G)}|D_yG(y_t^m,m)-D_yG(y_t^m,\tilde{m})|   
				\leq \frac{tL_1}{1-t\lambda_1(G)}W_1(m,\tilde{m})
		\end{equation}
		where $L_1$ is the Lipschitz constant of $D_y G$ with respect to $m$ in the $W_1$ metric. Therefore,
		\begin{equation}
			\int|y_t^m(x)-y_t^{\tilde{m}}(x)|\dif m_0(x)\leq \frac{tL_1}{1-t\lambda_1(G)}W_1(m,\tilde{m})
		\end{equation}
		So,
		\begin{equation}
				W_1(y_t^m\#m_0,y_t^{\tilde{m}}\#m_0)=\sup_{||\varphi||_{\Lip\leq 1}} \int\varphi \dif(y_t^m\#m_0-y_t^{\tilde{m}}\#m_0)
				\leq \frac{tL_1}{1-t\lambda_1(G)}W_1(m,\tilde{m})
		\end{equation}
		We deduce that for $t \leq t^*<\frac{1}{\lambda_1(G) + L_1}$, $F_{m_0}$ is a contraction on $\s{P}_1(\bb{R}^d)$

		\nextstep

		Set $\mu_0:=E(m_0)$ and $\mu_n:=E(m_{0,n})$.		
		By using the inverse function theorem on Equation \eqref{eq:critpt}, we deduce that for $0 < t \leq t^* < \frac{1}{\lambda_1(G) + L_1}$ we have $\nabla y_t^{\mu_n}(x) = (I + t D_{yy}^2 G(y_t^{\mu_n}(x),\mu_n))^{-1}$.
		Since $D^2_{yy}G$ is bounded we get $\enVert{\nabla y_t^{\mu_n}}<C_0$ such that $C_0$ does not depend on $n$, but only on $t$ and the bound on $D^2_{yy}G$.

		Let $\phi$ such that $\enVert{\nabla \phi}_\infty \leq 1$.  Then
		\begin{equation}
			\begin{split}\label{eq2}
				\int\phi d(\mu_0-\mu_n) &= \int\del{\phi\circ y^{\mu_0}_t-\phi\circ y^{\mu_n}} \dif m_0 + \int\phi\circ y_t^{\mu_n}\dif (m_0-m_n)\\
				&\leq \int \enVert{y^{\mu_0}_t-y^{\mu_n}_t} \dif m_0 + \int\phi\circ y_t^{\mu_n}\dif (m_0-m_n).
			\end{split}
		\end{equation}		
		Since $\enVert{\nabla \phi\circ y_t^{\mu_n}}_\infty \leq C_0$, by taking the supremum with respect to $\phi$ and applying \eqref{eq:ym-ym'} we have
		\begin{equation}
			W_1(\mu_0,\mu_n)\leq \frac{tL_1}{1-t\lambda_1(G)}W_1(\mu_0,\mu_n) + C_0 W_1(m_0,m_n).
		\end{equation}
		Since $t \leq t^*$ and $\frac{t^*L_1}{1-t^*\lambda_1(G)} < 1$, we let $n \to \infty$ to conclude that $\mu_n$ converges to $\mu_0$ respect to the $W_1$ metric.
	\end{proof}
	
	\subsection{Discretization} \label{sec:discrete}
	We will now discretize the problem by considering only empirical distributions as initial measures.
	We then give an algorithm, based on standard Picard iteration, to compute the equilibrium to any specified degree of accuracy.
	As empirical measures are dense in $\s{P}_1(\bb{R}^d)$, by part 3 of Theorem \ref{thm:well-posed}, our algorithm can in fact approximate the equilibrium for arbitrary initial measures to any specified degree of accuracy.
	
	Let us now suppose that the initial measure is an empirical distribution, i.e.~$m_0 = \sum_{j=1}^N a_j \delta_{x_j}$ for some points $x_1,\ldots,x_N \in \bb{R}^d$ and some non-negative numbers $a_j$ that sum to 1.
	Then
	\begin{equation}
		F_{m_0}(m) = y_t^m \sharp m_0 = \sum_{j=1}^N a_j \delta_{y_t^m(x_j)}.
	\end{equation}
	It follows that any equilibrium $m$ must itself be an empirical measure of the form $m = \sum_{j=1}^N a_j \delta_{z_j}$ for some points $z_1,\ldots,z_N \in \bb{R}^d$.
	For any vector ${\bf z} = (z_1,\ldots,z_N) \in \bb{R}^{dN}$, define
	\begin{equation}
		\begin{split}
			\tilde y_t^{\bf z}(x) &= y_t^m(x) \quad \text{where} \quad m = \sum_{j=1}^N a_j \delta_{z_j},\\
			\tilde F({\bf z}) &= \tilde y_t^{\bf z}({\bf x}) := (\tilde y_t^{\bf z}(x_1),\ldots,\tilde y_t^{\bf z}(x_N)).
		\end{split}
	\end{equation}
	To find an equilibrium, it is enough to find a fixed point of $\tilde F$.
	To see this, note that if ${\bf z} = \tilde F({\bf z})$, then we have
	\begin{equation}
		F_{m_0}\del{\sum_{j=1}^N a_j \delta_{z_j}} = \sum_{j=1}^N a_j \delta_{\tilde y_t^{\bf z}(x_j)}
		= \sum_{j=1}^N a_j \delta_{z_j}.
	\end{equation}
	
	It is not hard to see that under the same hypotheses as in Theorem \ref{thm:well-posed}, $\tilde F$ is a contraction, and it therefore has a unique fixed point, which we will denote
	\begin{equation} \label{eq:tilde E}
		\tilde E({\bf x}) := {\bf z} = \tilde F({\bf z}).
	\end{equation}
	By classical theory, one can define
	\begin{equation}
		\begin{split}
			{\bf z}_0 &= {\bf x},\\
			{\bf z}_{k+1} &= \tilde F({\bf z}_k)
		\end{split}
	\end{equation}
	and get ${\bf z}_k \to {\bf z} = \tilde F({\bf z})$.
	The error estimate is
	\begin{equation}
		\enVert{{\bf z}_k - {\bf z}} \leq \frac{\alpha^n}{1-\alpha}\enVert{\tilde F({\bf x}) - {\bf x}},
	\end{equation}
	where $\alpha \in (0,1)$ is the Lipschitz constant for $\tilde F$.
	\begin{remark}
		Although the contraction mapping theorem allows us to take any initial condition we like, we choose ${\bf z}_0 = {\bf x}$ for the following common sense reason.
		Since the time horizon is small, players cannot move much from their initial distribution, so the final distribution is sure to be close to ${\bf x}$; hence a good initial guess is ${\bf x}$ itself.
	\end{remark}
	In principle, this simple algorithm is enough to solve our problem.
	However, we cannot compute $\tilde F({\bf z})$ explicitly. 
	We will now point out that it is enough to approximate $\tilde F({\bf z})$ sufficiently well.
	\begin{lemma} \label{lem:approximate contraction}
		Suppose $\cbr{{\bf z}_k}$ is a sequence such that
		\begin{equation} \label{eq:almost recursion}
			\enVert{{\bf z}_{k+1} - \tilde F({\bf z}_k)} \leq \varepsilon_k,
		\end{equation}
		where $\varepsilon_k \leq \alpha^k$ and $\alpha < 1$ is the Lipschitz constant for $\tilde F$.
		Then ${\bf z}_k \to {\bf z}$, where ${\bf z}$ is the unique fixed point of $\tilde F$.
		We have an error estimate:
		\begin{equation} \label{eq:error estimate}
			\enVert{{\bf z}_k - {\bf z}} \leq \alpha^k\del{\frac{2\del{k-\alpha k+1}}{\del{1-\alpha}^2}+\frac{\enVert{\mathbf{z}_1-\mathbf{z}_0}}{1-\alpha}}.
		\end{equation}
	\end{lemma}
	
	\begin{proof}
		We have  $\enVert{\tilde F(\textbf{z}_{k+1})-\tilde F(\textbf{z}_{k})} \leq \alpha \enVert{\textbf{z}_{k+1}-\textbf{z}_{k}}$ for every $k$. Applying \eqref{eq:almost recursion} and the triangle inequality,
		\begin{align*}
			\enVert{\textbf{z}_{k+1}-\textbf{z}_{k}} &\leq \enVert{\textbf{z}_{k+1}-\tilde F(\textbf{z}_k)}+\enVert{\tilde F(\textbf{z}_{k})-\tilde F(\textbf{z}_{k-1})}+\enVert{\textbf{z}_{k}-\tilde F(\textbf{z}_{k-1})}\\
			&\leq \varepsilon_k+\alpha\enVert{\textbf{z}_{k}-\textbf{z}_{k-1}}+\varepsilon_{k-1}.
		\end{align*}
		By iteration we get
		\begin{equation}\label{eq:fixpt}
			\enVert{\textbf{z}_{k+1}-\textbf{z}_{k}} \leq \alpha^k \enVert{\textbf{z}_{1}-\textbf{z}_{0}} +\sum_{j=0}^{k}\alpha^{k-j}\del{\varepsilon_{j+1}+\varepsilon_{j}}.
		\end{equation}

		Now taking $m,n\in\bb{N}$ with $m>n$, we use \eqref{eq:fixpt} to estimate
		\begin{align*}
			\enVert{\mathbf{z}_m - \mathbf{z}_n} 
			\leq \sum_{k=n}^{m-1}\enVert{\mathbf{z}_{k+1} - \mathbf{z}_{k}}
			&\leq \sum_{k=n}^{m-1}\alpha^k\enVert{\mathbf{z}_{1} - \mathbf{z}_{0}} + \sum_{k=n}^{m-1}\sum_{j=0}^{k}\alpha^{k-j}\del{\varepsilon_{j+1}+\varepsilon_{j}}\\
			&\leq \frac{\alpha^n}{1-\alpha}\enVert{\mathbf{z}_{1} - \mathbf{z}_{0}} + \sum_{k=n}^{m-1}\sum_{j=0}^{k}\alpha^{k-j}\del{\varepsilon_{j+1}+\varepsilon_{j}}.
		\end{align*}
		Now we focus on the double sum. Using the assumption  $\varepsilon_k \leq \alpha^k$, we get
		\begin{equation*}
			\sum_{k=n}^{m-1}\sum_{j=0}^{k}\alpha^{k-j}\del{\varepsilon_{j+1}+\varepsilon_{j}}
			\leq \sum_{k=n}^{m-1}\del{k+1}\del{\alpha^{k+1}+\alpha^{k}}
			\leq 2\sum_{k=n}^{\infty}\del{k+1}\alpha^k.
		\end{equation*}	
		For a closed form we compute
		\begin{equation*}
			2\sum_{k=n}^{\infty}\del{k+1}\alpha^k=2\dod{}{\alpha}\sbr{\sum_{k=n}^\infty \alpha^{k+1}}
			=2\dod{}{\alpha}\sbr{\frac{\alpha^{n+1}}{1-\alpha}}
			=\frac{2\alpha^n\del{n-\alpha n+1}}{\del{1-\alpha}^2}.
		\end{equation*}
		We then deduce
		\begin{equation*}
			\enVert{\mathbf{z}_m-\mathbf{z}_n}\leq \alpha^n\del{\frac{2\del{n-\alpha n+1}}{\del{1-\alpha}^2}+\frac{\enVert{\mathbf{z}_1-\mathbf{z}_0}}{1-\alpha}}.
		\end{equation*}
		Therefore $\cbr{\mathbf{z}_k}$ is Cauchy and thus converges.
		The error estimate \eqref{eq:error estimate} is derived by taking $m \to \infty$ and replacing $n$ with $k$.
	\end{proof}
	
	Thanks to Lemma \ref{lem:approximate contraction}, we can now specify a fully formed algorithm to compute the equilibrium:
	\begin{itemize}
		\item Set ${\bf z}_0 = {\bf x}$.
		\item Given ${\bf z}_k$, use Newton's method to approximate $\tilde y_t^{{\bf z}_k}(x_j)$ for $j = 1,\ldots,N$, until we obtain a vector ${\bf z}_{k+1}$ satisfying
		\begin{equation}
			\enVert{{\bf z}_{k+1} - \tilde F({\bf z}_k)} \leq \alpha^k.
		\end{equation}
	\end{itemize}
	Then ${\bf z}_k$ will converge to the equilibrium, with error estimate given by Lemma \ref{lem:approximate contraction}.
	We remark that, since empirical measures are dense in the Wasserstein space and the equilibrium map $E(m)$ is continuous, this algorithm can also be used to approximate $E(m)$ for any $m \in \sr{P}_1(\bb{R}^d)$ within an arbitrary specified margin of error.
	
	\section{Asymptotic behavior of the equilibrium map} \label{sec:asymptotic}
	
	Recall that $E(m_0)$ is defined to be the equilibrium measure, i.e.~the final distribution, given an initial distribution $m_0$.
	In this section we study the dynamical system \eqref{eq:dynamical system}.
	We will focus in this paper on determining stability for the dynamical system \eqref{eq:dynamical system} for initial measures of the form $m_0=\sum_{j=1}^n \frac{1}{n}\delta_{x_j}$.
	Such empirical measures are also dense in the Wasserstein space $\sr{P}_1(\bb{R}^d)$.
	As discussed in Section \ref{sec:discrete}, for such initial measures we can replace $E(m)$ with $\tilde E({\bf x})$, so it is equivalent to study the classical dynamical system
	\begin{equation}\label{eq:Etilde}
		\begin{split}
			{\bf x}_0 &\in (\bb{R}^d)^n \quad \text{is given},\\
			{\bf x}_{k+1} &= \tilde E({\bf x}_k), \quad k = 0,1,2,\ldots
		\end{split}
	\end{equation}	
	We determine the asymptotic behavior by identifying
	\begin{itemize}
		\item the fixed points of $\tilde E$, and
		\item the spectrum of the derivative of $\tilde E$ at fixed points.
	\end{itemize}
	Our results in this section will hold for a particular case, namely when $d = 1$ and the cost function $G$ takes the form
	\begin{equation*}
		G(y,\mu)=\int \varphi(y-z)\dif \mu(z),
	\end{equation*}
	where $\varphi$ is a smooth function.
	However, we will start with some more general properties of $\tilde{E}$ before restricting to this special case.
	
	\subsection{General properties of $D\tilde E$}
	\begin{proposition}
		Recall the definition $y_t^\mu(x) =\operatorname{argmin} J_t(x,\cdot,\mu)$.
		Let $0 < t \leq t^*$, where $t^*$ is as in Theorem \ref{thm:well-posed}.
		Then $Dy_t^\mu(x)$ is a positive definite symmetric matrix for each $x$.
	\end{proposition}
	\begin{proof}
		Let $y(x) = y_t^\mu(x)$.
		Then
		\begin{equation}
			\frac{y(x)-x}{t}+D_yg(y(x),\mu)=0.
		\end{equation}
		Taking the derivative with respect to $x$ yields
		\begin{align}
			\frac{Dy(x)-I}{t}+D^2_{yy}G(y(x),\mu)Dy(x)&=0\\
			(I+tD^2_{yy}G(y(x),\mu))Dy(x)&=I
		\end{align}
		Note: $D^2_{yy}G(y(x),\mu)$ is a symmetric matrix and $I+tD^2_{yy}G(y(x),\mu)$ is as well.
		Moreover, by Assumption \ref{Gassum} and the condition $t \leq t^* < \frac{1}{\lambda_1(G)}$, we see that $I+tD^2_{yy}G(y(x),\mu)$ is positive definite, hence its inverse $Dy(x)$ is positive definite as well. 
	\end{proof}
	Now we may use this result to determined the fixed points of the dynamical system.
	\begin{proposition}\label{prop:fixpt}
		Let 
		\begin{equation}
			\mu=\frac{1}{n}\sum_{i=1}^n \delta_{x_i}
		\end{equation}
		where $x_i\in \mathbb{R}$ for $1\leq i\leq n$. Suppose $y\#\mu=\mu$, and $y$ is strictly monotone increasing, then $y(x)=x$ for all $x\in \operatorname{supp}(\mu)$.
	\end{proposition}
	\begin{proof}
		Since $\mu$ is an empirical measure, $\operatorname{supp}(\mu)=\{x_i\}_{i=1}^n$. Also, since $y\#\mu=\mu$, 
		\begin{equation}
			\operatorname{supp}(y\#\mu)=\operatorname{supp}(\mu)=\{x_i\}_{i=1}^n
		\end{equation}
		Therefore for each $1\leq i\leq n$, $y(x_i)=x_j$ for some $1\leq j\leq n$. Without loss of generality we may assume $x_1<x_2<\dots<x_n$. Suppose $y(x_1)=x_i$ for some $i>1$. Then since $y$ is increasing, $y(x_j)>x_i$ for all $j>1$. Then by the pigeonhole principle, this implies $y$ is not injective.
		This contradicts that $y$ is strictly monotone increasing, so $y(x_1)=x_1$. This argument can be repeated to show $y(x_i)=x_i$ for each $1\leq i\leq n$. Therefore $y(x)=x$ for all $x\in \operatorname{supp}(\mu)$. 
	\end{proof}
	Now we know fixed points of $\tilde E$ correspond to measures $\mu = \frac{1}{n}\sum_{j=1}^n x_j$ such that $y_t^\mu=I$ on the support of $\mu$.
	Our next goal is to linearize $\tilde E$ around fixed points to check if the system is stable. We do this by computing $D\tilde E(\mathbf{x})$ and showing that the eigenvalues of $D\tilde E(\mathbf{x})$ are small if $\mathbf{x}$ is a fixed point.
	
	Let us introduce the notation
	\begin{equation*}
		\mu_{\mathbf{x}}=\frac{1}{n}\sum_{j=1}^n x_j \quad \forall \mathbf{x}=(x_1,\dots,x_n).
	\end{equation*}
	Then we define $g(y,\mathbf{x})=G(y,\mu_{\mathbf{x}})$.
	Now $\mu_{\mathbf{y}}$ is the equilibrium for an initial measure $\mu_{\mathbf{x}}$ if and only if for each $1\leq j\leq n$, $y_j$ is the optimal move for $x_j$ given $\mu_{\mathbf{y}}$ as a final measure, i.e.
	\begin{equation}\label{equilibrium}
		y_j+tD_yg(y_j(\mathbf{x}),\mathbf{y})=x_j.
	\end{equation}
	Hence we can define
	\begin{equation}
		\mathbf{y}(\mathbf{x}) = (y_1(\mathbf{x}),\ldots,y_n({\bf x})) = \tilde{E}({\bf x})
	\end{equation}
	implicitly through \eqref{equilibrium}.
	
	We want to compute $D\tilde E(\mathbf{x}) = D{\bf y}({\bf x})$. To do this we take the implicit partial derivative of \eqref{equilibrium} with respect to $x_k$ for $1\leq k\leq n$
	\begin{equation}
		\frac{\partial}{\partial x_k}[y_j(\mathbf{x})+tD_yg(y_j(\mathbf{x}),\mathbf{y}(\mathbf{x}))=x_j]
	\end{equation}
	to get
	\begin{equation}
		\frac{\partial y_j}{\partial x_k}+tD_{yy}^2g(y_j(\mathbf{x}),\mathbf{y}(\mathbf{x}))\frac{\partial y_j}{\partial x_k}+t\sum_{i=1}^n D_{yx_i}^2g(y_j(\mathbf{x}),\mathbf{y}(\mathbf{x}))\cdot \frac{\partial y_i}{\partial x_k}=\frac{\partial x_j}{\partial x_k}=\delta_{j,k}.
	\end{equation}
	We get a system of equations which is presented in matrix form as $A D\tilde E(\mathbf{x}) = I$ where $A=(A_{i,j})_{i,j=1}^n$,
	\begin{equation}
		A_{i,j}=\delta_{i,j}(1+tD^2_{yy}g(y_j(\mathbf{x}),\mathbf{y}(\mathbf{x})))+tD_{y x_i}^2g(y_j(\mathbf{x}),\mathbf{y}(\mathbf{x})).
	\end{equation}
	Next we will determine the spectrum of $A$ and thereby obtain the spectrum of $D\tilde E(\mathbf{x})$.
	
	\subsection{Spectral properties}
	For simplicity, we will look specifically at the case $$G(y,\mu)=\int \varphi(y-z)\dif \mu(z).$$
	In this case,
	\begin{equation}
		g(y,\mathbf{y})=G\del{y,\frac{1}{n}\sum_{k=1}^n\delta_{y_k}}=\frac{1}{n}\sum_{k=1}^n\varphi(y-y_k).
	\end{equation}
	Then we compute
	\begin{comment}
	\begin{equation}
		D^2_{yy}g(y,\mathbf{y})=\frac{1}{n}\sum_{k=1}^n\varphi''(y-y_k) 
	\end{equation}
	and at $y=y_j$
	\end{comment}
	\begin{equation}
		D^2_{yy}g(y_j,\mathbf{y})=\frac{1}{n}\sum_{k=1}^n\varphi''(y_j-y_k) 
	\end{equation}
	and 
	\begin{equation}
		D^2_{yx_i}g(y_j,\mathbf{y})=-\frac{1}{n}\varphi''(y_j-y_k).
	\end{equation}
	\begin{comment}
	So
	\begin{equation}
		\begin{split}
			A_{i,i}&=1+\frac{t}{n}\sum_{k=1}^n\varphi''(y_i-y_k)-\frac{t}{n}\varphi''(y_i-y_i)\\
			&=1+\frac{t}{n}\sum_{k=1,k\not=i}^n\varphi''(y_i-y_k)
		\end{split}
	\end{equation}
	And for $i\not=j$
	\begin{equation}
		A_{i,j}=-\frac{t}{n}\varphi''(y_j-y_i)
	\end{equation}
	\end{comment}
	We see $A$ can be written in the following form
	\begin{equation}
		\begin{split}
			A&=I+\frac{t}{n}\begin{pmatrix}
				\sum_{k=1,k\not=1}^n\varphi''(y_1-y_k)  & -\varphi''(y_2-y_1) & \cdots &-\varphi''(y_n-y_1)\\
				-\varphi''(y_1-y_2) & \sum_{k=1,k\not=2}^n\varphi''(y_2-y_k) & \cdots &-\varphi''(y_n-y_1)\\
				\vdots & \vdots & \ddots & \vdots\\
				-\varphi''(y_1-y_n) & -\varphi''(y_2-y_n) & \cdots & \sum_{k=1,k\not=n}^n\varphi''(y_n-y_k)
			\end{pmatrix}\\
			&=: I+\frac{t}{n}B
		\end{split}    
	\end{equation}
    \begin{theorem}\label{thm:fixpoint}
        Suppose $\varphi$ is even, has a minimum at $x=0$, is increasing on the interval $\del{0,r}$, and $\varphi(x)=0$ for $x\geq r$. Then $\mathbf{x}=(x_1,x_2,\dots,x_n)$ is a fixed point of $\tilde E$ if and only if for any $1\leq j,k \leq n$, $x_j=x_k$ or $|x_j-x_k|\geq r$.
    \end{theorem}
    \begin{proof}
        Note: we may arrange $\mathbf{x}$ so that $x_1\leq x_2 \leq \dots \leq x_n$. Now if $\mathbf{x}$ is a fixed point of $\tilde E$, then for each $1\leq j\leq n$
        \begin{equation}
            y_j+\frac{t}{n}\sum_{k=1}^n\varphi'(y_j-y_k)=x_j,
        \end{equation}
        and by Proposition \ref{prop:fixpt} $y_j=x_j$, so for $1\leq j\leq n$
        \begin{equation}
           \sum_{k=1}^n\varphi'(x_j-x_k)=0.
        \end{equation}
        For $j=1$, $\sum_{k=1}^n\varphi'(x_k-x_1)=0$. Now if $x_1=x_k$ for any $k$, $\varphi'(x_k-x_1)=0$. Suppose $k_1+1$ is the first $k$ such that $x_k-x_1>0$, then 
        \begin{equation}
           \sum_{k=k_1+1}^n\varphi'(x_j-x_k)=0 
        \end{equation}
        and since $\varphi'(x_k-x_1)\geq 0$ for $k> k_1$, we need $|x_k-x_1|\geq r$ for $k>k_1$. So $x_1=x_k$ for $k\leq k_1$ and $|x_k-x_1|\geq r$ for $k> k_1$.

        Now suppose $x_j$ is such that $|x_j-x_k|\geq r$ for $k<j$, then
        \begin{equation}
            \sum_{k=1}^j \varphi'(x_k-x_j)=0.
        \end{equation}
        So we still need 
        \begin{equation}
            \sum_{k=j+1}^n\varphi'(x_k-x_j)=0.
        \end{equation}
        If $x_j=x_k$ for any $k>j$, $\varphi'(x_k-x_j)=0$, and if $k_j+1$ is the first $k$ such that $x_k-x_j>0$,
        \begin{equation}
            \sum_{k=k_j+1}^n\varphi'(x_k-x_j)=0.
        \end{equation}
        Therefore $\varphi'(x_k-x_j)=0$ for $k> k_j$, so $|x_k-x_j|\geq r$ for $k>k_i$. Thus our hypothesis holds.

        Now suppose for any $1\leq j, k\leq n$ either $x_k=x_j$ or $|x_k-x_j|>r$. Then if $x_j=x_k$, $\varphi'(x_j-x_k)=\varphi'(0)=0$, and if $|x_k-x_j|\geq r$, $\varphi'(x_k-x_j)=0$. So 
        \begin{equation}
            \sum_{k=1}^n\varphi'(x_k-x_j)=0
        \end{equation}
        for all $1\leq j\leq n$. Thus $\mathbf{x}$ is a fixed point for $\tilde E$.
    \end{proof}
	Let ${\bf x}$ be a fixed point of $\tilde{E}$.
    Since we have this clustering behavior from Theorem \ref{thm:fixpoint}, it will be helpful to introduce the following notation:
    \begin{align*}
        x_{k_1}&=x_k, &1=k_0\leq &k\leq k_1\\
        x_{k_2}&=x_k, &k_1< &k\leq k_2\\
        & \vdots & &\vdots\\
        x_{k_{i}}&=x_k, &k_{i-1}< &k\leq k_i=n
    \end{align*}
	where, without loss of generality, we have labeled the components of ${\bf x}$ in increasing order.
	By Theorem \ref{thm:fixpoint}, we have $x_{k_{m}}-x_{k_{m-1}}\geq r$ for $2\leq m\leq i$.
	\begin{definition} \label{def:spread}
		We define ${\bf x}$ to be \emph{sufficiently spread out} if we have the strict inequality $x_{k_{m}}-x_{k_{m-1}}>r$ for $2\leq m\leq i$.
	\end{definition}
    
    Recall, we are evaluating $D\tilde E(\mathbf{x})$ for $\mathbf{x}$ a fixed point of $\tilde E$. So it suffices to look at the spectrum of $B$ evaluated at fixed points. According to Theorem \ref{thm:fixpoint}, if $x_k\not=x_j$, then $|x_j-x_k|>r$ and $\varphi''(x_j-x_k)=0$. But if $x_k=x_j$, then $\varphi''(x_j-x_k)=\varphi''(0)=1$.
    \begin{equation}
        B=\begin{pmatrix}
				\sum_{k=1,k\not=1}^n\varphi''(x_1-x_k)  & -\varphi''(x_2-x_1) & \cdots &-\varphi''(x_n-x_1)\\
				-\varphi''(x_1-x_2) & \sum_{k=1,k\not=2}^n\varphi''(x_2-x_k) & \cdots &-\varphi''(x_n-x_1)\\
				\vdots & \vdots & \ddots & \vdots\\
				-\varphi''(x_1-x_n) & -\varphi''(x_2-x_n) & \cdots & \sum_{k=1,k\not=n}^n\varphi''(x_n-x_k)
			\end{pmatrix}
    \end{equation}
    We see $B$ has a block diagonal form:
    \begin{equation}
        B=\begin{pmatrix}
            B_1 & 0 & \cdots & 0\\
            0 & B_2 & \cdots & 0\\
            \vdots & \vdots &\ddots & \vdots\\
            0 & 0 & \cdots & B_i
        \end{pmatrix}
    \end{equation}
    where for $1\leq m\leq i$, $B_m$ is the $(k_m-k_{m-1})\times(k_m-k_{m-1})$ matrix given by
    \begin{equation}
        B_m=\begin{pmatrix}
            (k_m-k_{m-1}-1) & -1 & \cdots & -1\\
            -1 & (k_m-k_{m-1}-1) & \cdots & -1\\
            \vdots & \vdots &\ddots & \vdots\\
            -1 & -1 & \cdots & (k_m-k_{m-1}-1)
        \end{pmatrix}.
    \end{equation}
	\begin{lemma}\label{lem:Bpos}
	    Let $C_n$ be an $n\times n$ matrix with the following form:
     \begin{equation}
         C_n=\begin{pmatrix}
            n-1 & -1 & \cdots & -1\\
            -1 & n-1 & \cdots & -1\\
            \vdots & \vdots &\ddots & \vdots\\
            -1 & -1 & \cdots & n-1
        \end{pmatrix}.
     \end{equation}
     Then if $n\geq 2$, $C_n$ is a positive matrix.
	\end{lemma}
 \begin{proof}
    We proceed by induction on $n$.\\
		Base case: $n=2$
		\begin{equation}
			C_2=\begin{pmatrix}
				1 & -1\\
				-1 & 1
			\end{pmatrix}
		\end{equation}
		Let $\mathbf{v}=\begin{pmatrix}
			v_1\\
			v_2
		\end{pmatrix}$ be an arbitrary real vector. Then 
		\begin{equation}
			\mathbf{v}^T  C_2 \mathbf{v}=(v_1-v_2)^2\geq0.
		\end{equation}
		Now supposing the hypothesis holds for $C_{n-1}$ for $n>2$, we want to prove it holds for $C_n$. Let $\mathbf{v}=\begin{pmatrix}
			v_1\\
			\vdots\\
			v_n
		\end{pmatrix}$. Now,
		\begin{equation}\label{eq:decomp}
			\begin{split}
				\mathbf{v}^TC_n\mathbf{v}&=\begin{pmatrix}
					v_1 &\cdots & v_{n-1}
				\end{pmatrix} C_{n-1} \begin{pmatrix}
					v_1\\
					\vdots \\
					v_{n-1}
				\end{pmatrix}-2v_n\sum_{k=1}^{n-1}v_k+\sum_{k=1}^{n-1}v_k^2+\sum_{k=1}^{n-1}v_n^2\\
				&=\begin{pmatrix}
					v_1 &\cdots & v_{n-1}
				\end{pmatrix} C_{n-1} \begin{pmatrix}
					v_1\\
					\vdots \\
					v_{n-1}
				\end{pmatrix}+\sum_{k=1}^{n-1}(v_k-v_n)^2  
			\end{split} 
		\end{equation}
		By our inductive hypothesis, 
		\begin{equation*}
			\begin{pmatrix}
				v_1 &\cdots & v_{n-1}
			\end{pmatrix} C_{n-1} \begin{pmatrix}
				v_1\\
				\vdots \\
				v_{n-1}
			\end{pmatrix}\geq 0.
		\end{equation*}
		Additionally, $(v_k-v_n)^2\geq 0$ for all $1\leq k\leq n-1$, so we can conclude $C_n$ is a positive matrix.
 \end{proof}
 
	\begin{corollary}\label{cor:eig0}
		If $n\geq 2$, the eigenspace for $C_n$ for the eigenvalue $\lambda=0$ is one dimensional and spanned by $\begin{pmatrix}
			1\\
			1\\
			\vdots\\
			1
		\end{pmatrix}$.
	\end{corollary}
	\begin{proof} 
		\begin{equation}
			C_n\begin{pmatrix}
				1\\
				1\\
				\vdots\\
				1
			\end{pmatrix}=\begin{pmatrix}
            n-1 & -1 & \cdots & -1\\
            -1 & n-1 & \cdots & -1\\
            \vdots & \vdots &\ddots & \vdots\\
            -1 & -1 & \cdots & n-1
        \end{pmatrix}\begin{pmatrix}
				1\\
				1\\
				\vdots\\
				1
			\end{pmatrix}=\mathbf{0}
		\end{equation} 
		So our vector is an eigenvector for the eigenvalue $\lambda=0$. Now we'd like to show that it spans the eigenspace. Without loss of generality, suppose $v_n\not=v_k$ for some $1\leq k\leq n-1$ where $\mathbf{v}=\begin{pmatrix}
			v_1\\
			v_2\\
			\vdots\\
			v_n
		\end{pmatrix}$. Then from \eqref{eq:decomp} we have 
		\begin{equation}
			\mathbf{v}^TC_n\mathbf{v}=\begin{pmatrix}
				v_1 &\cdots & v_{n-1}
			\end{pmatrix} C_{n-1} \begin{pmatrix}
				v_1\\
				\vdots \\
				v_{n-1}
			\end{pmatrix}+\sum_{k=1}^{n-1}(v_k-v_n)^2 
		\end{equation}
		As before, the first piece is nonnegative and since $v_n\not=v_k$ for some $k$, the sum is strictly greater than $0$ so any $\mathbf{v}$ not in the span of $\begin{pmatrix}
			1\\
			1\\
			\vdots\\
			1
		\end{pmatrix}$ is not an eigenvector for $\lambda=0$.
	\end{proof}

	So eigenvalues for $B$ are greater than or equal to $0$. Recalling $D\tilde E(\mathbf{x})=A^{-1}$ and $A=I+\frac{t}{n}B$, eigenvalues for $D\tilde E(\mathbf{x})$ are $\lambda\leq 1$. Now for stability analysis, it can be problematic for $D\tilde E (\mathbf{x})$ to have an eigenvalue of $1$, but when the eigenspace is spanned by $\begin{pmatrix}
		1\\
		1\\
		\vdots\\
		1
	\end{pmatrix}$, we can decompose our empirical measures in the system using the following proposition.
	
	\begin{proposition}\label{prop:translation}
		If $g(y_j+\alpha,\mathbf{y}+\alpha(1,\dots,1))=g(y_j,\mathbf{y})$ for any $\alpha\in\bb{R}$ and $y\in\bb{R}^n$, then \mbox{$\tilde E(\mathbf{x}+\alpha(1,\dots,1))=\tilde E(\mathbf{x})+\alpha(1,\dots,1)$}.
	\end{proposition}
	\begin{proof}
		Recall: $\tilde E(\mathbf{x})=\mathbf{y}$ if and only if for all $1\leq j\leq n$
		\begin{equation}
			y_j(\mathbf{x})+tD_yg(y_j(\mathbf{x}),\mathbf{y}(\mathbf{x}))=x_j.   
		\end{equation}
		So if $\tilde E(\mathbf{x}+\alpha(1,\dots,1))=\mathbf{z}$ then
		\begin{equation}
			z_j+tD_y g(z_j,\mathbf{z})=x_j+\alpha   
		\end{equation}
		Equivalently,
		\begin{equation}
			z_j-\alpha+tD_y g(z_j,\mathbf{z})=x_j
		\end{equation}
		Now since $g(y+\alpha,\mathbf{y}+\alpha(1, \dots, 1))=g(y,\mathbf{y})$
		\begin{equation}
			z_j-\alpha+tD_y g(z_j-\alpha,\mathbf{z}-\alpha(1,\dots,1))=x_j 
		\end{equation}
		This implies $\tilde E(\mathbf{x})=\mathbf{z}-\alpha(1,\dots,1)$, so $\tilde E(\mathbf{x}+\alpha(1,\dots,1))=\tilde E(\mathbf{x})+\alpha(1,\dots,1)$ as desired.
	\end{proof}
	\begin{remark}
		$g(y,\mathbf{y})=\frac{1}{n}\sum_{k=1}^n\varphi(y-y_k)$ meets the criteria of Proposition \ref{prop:translation}.
	\end{remark}
	Let
	\begin{equation}
		W=[\text{span}\{(1,\dots,1)\}]^\perp 
	\end{equation}
	Now, let $\mathbf{x}$ be a point in our system. Then $\mathbf{x}$ has the following orthogonal decomposition
	\begin{equation}
		\mathbf{x}=\mathbf{w}+\alpha (1,\dots, 1)
	\end{equation}
	where $\alpha\in \bb{R}$ and $\mathbf{w}\in W$. This allows us to restrict the domain of the dynamical system to $W$. By Proposition \ref{prop:translation}, we have that $\tilde E:W \to W$.
	\begin{theorem}\cite[Theorem 1.3.7]{stuart1998dynamical}\label{thm:stability}
		Suppose 
		\begin{equation}
			\begin{split}
				\mathbf{y}_{k+1}&=F(\mathbf{y}_k)\\
				\mathbf{y}_0&\in \bb{R}^n
			\end{split}
		\end{equation}
		where $F:D\to \bb{R}^n$ and $D\subset \bb{R}^n$. Let $F\in \s{C}^2(\bb{R}^n,\bb{R}^n)$, then an equilibrium point $\mathbf{y}^*$ of the system is asymptotically stable if the eigenvalues of $DF(\mathbf{y}^*)$ lie strictly inside the unit circle.
	\end{theorem}

	\begin{theorem}\label{thm:stability2}
		Suppose
		\begin{equation}\label{eq:system}
			\begin{split}
				\mathbf{x}_{k+1}&=\tilde E(\mathbf{x}_{k})\\
				\mathbf{x}_0&\in W
			\end{split}
		\end{equation}
		where $\tilde E: W\to W$ is as defined in \eqref{eq:Etilde}. Then an equilibrium point $\mathbf{x}^*\in\operatorname{span}\cbr{(1,\dots,1)}$ of the system is asymptotically stable. 
	\end{theorem}
	\begin{proof}
		We just need to show \eqref{eq:system} satisfies the assumptions of Theorem \ref{thm:stability}. $\tilde E\in \s{C}^2(\bb{R}^n,\bb{R}^n)$ by a routine application of the Implicit Function Theorem, and therefore by restricting to $W$ we have $\tilde E\in \s{C}^2(W,W)$, where $W$ is isomorphic to $\bb R^{n-1}$. By \ref{lem:Bpos}, $D\tilde E(\mathbf{x}^*)$ has eigenvalues $\lambda\leq1$. However, each $\bf x\in\bb R^n$ has an orthogonal decomposition of the form 
  \begin{equation}
     \mathbf{x}= \mathbf{w} +\alpha (1,\dots, 1) 
  \end{equation} and by Proposition \ref{prop:translation}
  \begin{equation}
      \tilde E(\mathbf{x})=\tilde E(\mathbf{w})+\alpha(1,\dots,1)
  \end{equation}
  Therefore $D\tilde E(\mathbf{x}^*)$ can be decomposed into one part that acts on $W$ and the other part that acts on $W^\perp$, where $W^\perp$ is the eigenspace corresponding to eigenvalue $1$. So $D\tilde E(\mathbf{x}^*)$ restricted to $W$ has eigenvalues $\lambda<1$, thus an equilibrium point $\mathbf{x}^*$ is asymptotically stable.
	\end{proof}

 \begin{corollary}\label{cor:stability}
   		Suppose
		\begin{equation}
			\begin{split}
				\mathbf{x}_{k+1}&=\tilde E(\mathbf{x}_{k})\\
				\mathbf{x}_0&\in \bb{R}^n
			\end{split}
		\end{equation}
		where $\tilde E: \bb{R}^n\to \bb{R}^n$ is as defined in \eqref{eq:Etilde}. Then an equilibrium point $\mathbf{x}^*\in\operatorname{span}\cbr{(1,\dots,1)}$ of the system is asymptotically stable.   
 \end{corollary}
 
 \begin{proof}
For any $x\in\mathbb{R}^n$, $x=w+\alpha(1,\dots,1)$ for $w\in W$ and $\alpha\in\mathbb{R}$. We see $\alpha(1,\dots,1)$ simply produces a translation of $\tilde E$ by $\alpha$, which does not affect asymptotic stability. By Theorem \ref{thm:stability2}, an equilibrium point of the form $\mathbf{x}^*\in\operatorname{span}\cbr{(1,\dots,1)}$ is asymptotically stable for $\tilde E$ restricted to $W$, so it is also stable on $\tilde E:\mathbb{R}^n\to\mathbb{R}^n$.
 \end{proof}
 Because of asymptotic stability, with each iteration, points get closer to the equilibrium. So if $\tilde E(\mathbf{x})=\mathbf{y}$ and $|\mathbf{x}-\mathbf{x}^*|$ is small enough, it should be true that $|\mathbf{y}-\mathbf{x}^*|<|\mathbf{x}-\mathbf{x}^*|$. We see this by the following:
 \begin{equation}
     \mathbf{y}-\mathbf{x}^*=D\tilde E(\mathbf{x}^*)(\mathbf{x}-\mathbf{x}^*)+o(|\mathbf{x}-\mathbf{x}^*|)
 \end{equation}
 Without loss of generality we look at orthogonal components $\mathbf{y}-\mathbf{x}^*\perp \mathbf{x}^*$,
 \begin{equation}
     \begin{split}
         |\mathbf{y}-\mathbf{x}^*|^2&=\langle\mathbf{y}-\mathbf{x}^*,\mathbf{y}-\mathbf{x}^*\rangle\\
         &\leq\langle D\tilde E(\mathbf{x}^*)(\mathbf{x}-\mathbf{x}^*),\mathbf{y}-\mathbf{x}^*\rangle+\langle\varepsilon(\mathbf{x}-\mathbf{x}^*),\mathbf{y}-\mathbf{x}^*\rangle\\
         &\leq \lambda_{\max}|\mathbf{x}-\mathbf{x}^*||\mathbf{y}-\mathbf{x}^*|+\varepsilon|\mathbf{x}-\mathbf{x}^*||\mathbf{y}-\mathbf{x}^*|
     \end{split}
 \end{equation}
 So we see
 \begin{equation}
     |\mathbf{y}-\mathbf{x}^*|\leq(\lambda_{\max}+\varepsilon)|\mathbf{x}-\mathbf{x}^*|.
 \end{equation}
Noting that $\lambda_{\max}<1$ and $\varepsilon$ can be made arbitrarily small yields
\begin{equation}\label{eq:fixpt estimate}
    |\mathbf{y}-\mathbf{x}^*|<|\mathbf{x}-\mathbf{x}^*|
\end{equation}
as desired.

Earlier we defined $\tilde E(\mathbf{x})$ for a fixed $n$ and $t\leq t^*$. Now we define for $\mathbf{x}\in \mathbf{R}^n$
\begin{equation}
    \tilde E^t_n(\mathbf{x}):=\mathbf{z}=\tilde y_t^\mathbf{z}(\mathbf{x})
\end{equation}

\begin{lemma}\label{lem: game splitting}
    Suppose $\mathbf{x}$ is sufficiently close to a fixed point $\mathbf{x}^*$, where ${\bf x}^*$ is sufficiently spread out (Definition \ref{def:spread}), so that
    \begin{equation*}
      x_1\leq \dots\leq x_{k_1}\leq x_{k_1+1}\leq \dots\leq x_{k_2}\leq \dots \leq x_{k_i}  
    \end{equation*} 
    are such that $x_{k_m+1}-x_{k_m}>r$ for $1\leq m\leq i-1$. Then 
    \begin{equation}
        \tilde E_n^t(\mathbf{x})=(\tilde E_{k_1}^{t_1}(x_1,\dots,x_{k_1}),\tilde E_{k_2-k_1}^{t_2}(x_{k_1+1},\dots, x_{k_2}),\dots, \tilde E_{k_i-k_{i-1}}^{t_i}(x_{k_{i-1}+1},\dots, x_{k_i}))
    \end{equation}
    where $t_m=\frac{t(k_{m+1}-k_m)}{n}$.
\end{lemma}
\begin{proof}
    Let us consider $\tilde E_{k_{m+1}-k_m}^{t_m}(x_{k_m+1},\dots,x_{k_{m+1}})$ for $0\leq m\leq i-1$ and $k_i=n$.
    \begin{equation}
       \tilde E_{k_{m+1}-k_m}^{t_m}(x_{k_m+1},\dots,x_{k_{m+1}})=(y_{k_m+1},\dots,y_{k_{m+1}})
    \end{equation}
    if and only if, for each $k_m+1 \leq j\leq k_{m+1}$
    \begin{equation}
        y_j+\frac{t_m}{k_{m+1}-k_m}\sum_{k=k_m+1}^{k_{m+1}}\varphi'(y_j-y_k)=x_j.
    \end{equation}
    Now we show $y_{k_m+1}-y_{k_m}>r$. Using \eqref{eq:fixpt estimate} and that $x_{k_m}\leq x_{k_{m+1}}^*\leq x_{k_{m+1}}$,
    \begin{equation}
    \begin{split}
       y_{k_m+1}-y_{k_m}&= y_{k_m+1}-x_{k_m+1}^*+x_{k_m+1}^*-x_{k_m}^*+x_{k_m}^*-y_{k_m}\\
       &\geq x_{k_m+1}-x_{k_m+1}^*+x_{k_m+1}^*-x_{k_m}^*+x_{k_m}^*-x_{k_m}\\
       &=x_{k_m+1}-x_{k_m}>r.
    \end{split}
    \end{equation}
    Thus $\varphi'(y_j-y_k)=0$ for $k\geq k_{m+1}+1$ and $k\leq k_m$. So we can write
    \begin{equation}
        y_j+\frac{t_m}{k_{m+1}-k_m}\sum_{k=1}^{n}\varphi'(y_j-y_k)=x_j.
    \end{equation}
    Now taking $\frac{t_m}{k_{m+1}-k_m}=\frac{t}{n}$ yields
    \begin{equation}
        y_j+\frac{t}{n}\sum_{k=1}^{n}\varphi'(y_j-y_k)=x_j.
    \end{equation}
    and this is true if and only if $\tilde E_n^t(x_1,\dots, x_n)=(y_1,\dots, y_n)$.
\end{proof}

This allows us to look at asymptotic behavior for each $\tilde E_{(j_{m}-j_{m-1})}^{t_m}$, which allows us to prove the main result of this section:
\begin{theorem}
    Suppose
	\begin{equation}
		\begin{split}
			\mathbf{x}_{k+1}&=\tilde E_n^t(\mathbf{x}_{k})\\
			\mathbf{x}_0&\in \bb{R}^n
		\end{split}
	\end{equation}
 Then every sufficiently spread out (see Definition \ref{def:spread}) equilibrium point $\mathbf{x}^*$ of the dynamical system is asymptotically stable. 
\end{theorem}
\begin{proof}
  Let $\mathbf{x}^*$ be an equilibrium. Then it has the following form
    \begin{equation*}
      x^*_1= \dots= x^*_{k_1}< x^*_{k_1+1}= \dots= x^*_{k_2}< \dots< x^*_{k_i-1} = x^*_{k_i}  
    \end{equation*} 
    with $x^*_{k_m+1}-x^*_{k_m}>r$.
  Using Lemma \ref{lem: game splitting}, for $|\mathbf{x}-\mathbf{x}^*|<\varepsilon$
  \begin{equation}
     \tilde E_n^t(\mathbf{x})=(\tilde E_{k_1}^{t_1}(x_1,\dots,x_{k_1}),\tilde E_{k_2-k_1}^{t_2}(x_{k_1+1},\dots, x_{k_2}),\dots, \tilde E_{k_i-k_{i-1}}^{t_i}(x_{k_{i-1}+1},\dots, x_{k_i}))  
  \end{equation}
  Then by Corollary \ref{cor:stability}, an equilibrium point $\mathbf{x}^{m*}\in\operatorname{span}\cbr{(1,\dots,1)}$ of the dynamical system with $\tilde E_{(j_{m}-j_{m-1})}^{t_m}$ is asymptotically stable. So $\mathbf{x}^*=(\mathbf{x}^{1*},\mathbf{x}^{2*},\dots,\mathbf{x}^{i*})$ is asymptotically stable for the dynamical system with $\tilde E_n^t$.
\end{proof}
	
	\section{Numerical simulations} \label{sec:numerics}
	
	In this section we report on some numerical simulations, based on the algorithm outlined at the end of Section \ref{sec:computation}, that confirm the main result of Section \ref{sec:asymptotic}.
	We examined four different empirical population measures distributed across a vector of 1000 evenly spaced points between $[-4.995, 4.995]$.
	The coupling function was given to be $G(x,m) = \int \varphi(x-z)\dif m(z)$ where $-\varphi$ is a standard ``bump function'' of the form
	\begin{equation} \label{eq:bump}
		-\varphi(x) = \begin{cases}
			\exp\del{-\frac{1}{1-x^2}}, \quad -1 < x < 1,\\
			0, \quad \text{otherwise}.
		\end{cases}
	\end{equation}
	\begin{center} 
		\includegraphics[width=.8\textwidth]{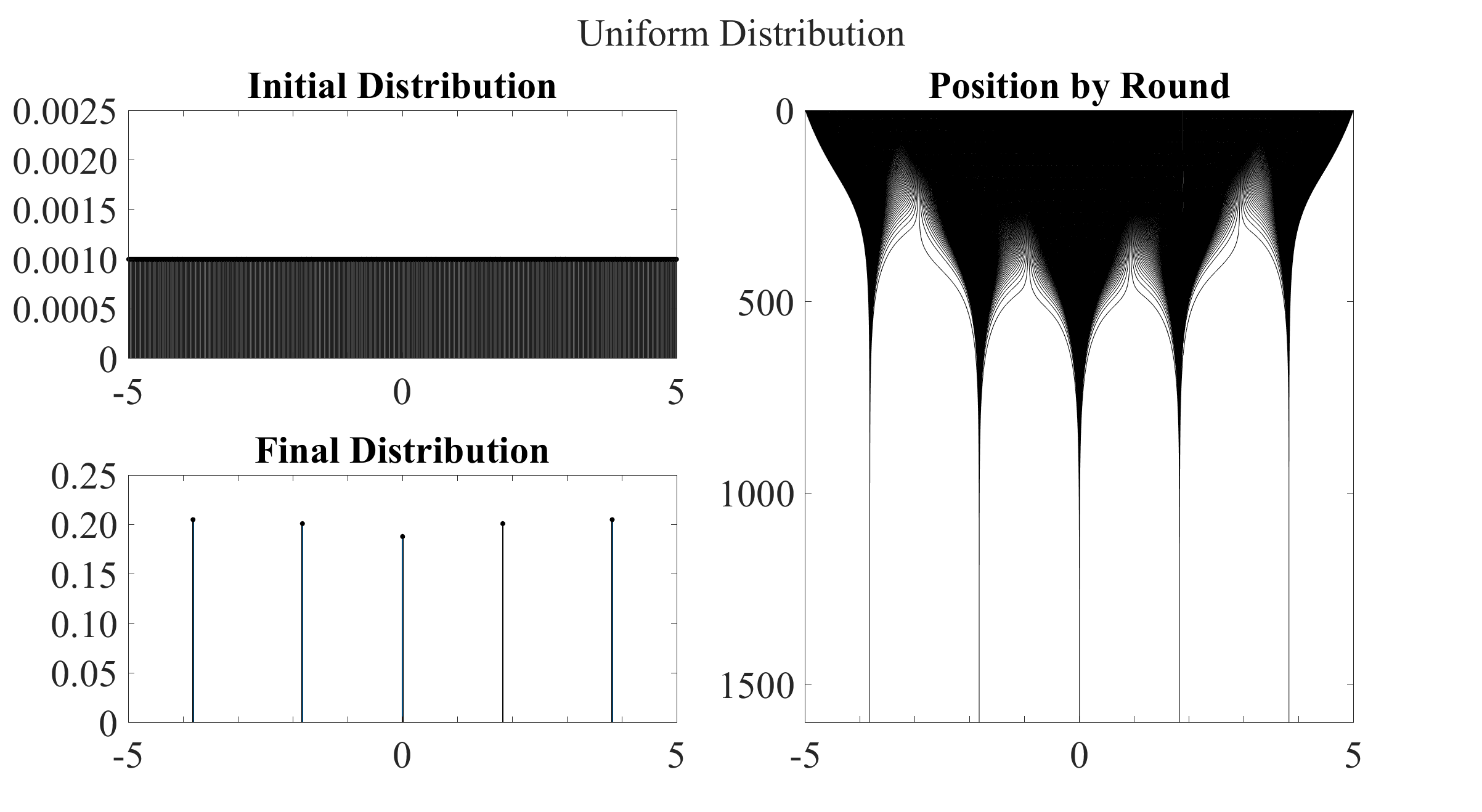}
		\begin{tabular}{|M{.06\textwidth}|M{.10\textwidth}|M{.12\textwidth}|M{.12\textwidth}|M{.18\textwidth}|M{.09\textwidth}|M{.11\textwidth}|M{.11\textwidth}|}
			\hline
			Group & Coalescing Point & Algorithm Iterations to Coalesce & Total \% of Population & Range of Initial Positions & Average Initial Position & Average Total Movement & Average Total Cost \\
			\hline
			1 & -3.8225 & 1260 & 20.50\% & [-4.9950, -2.9550] & -3.9750 & 0.5335 & -241.9120\\
			\hline
			2 & -1.8275 & 1417 & 20.10\% & [-2.9450, -0.9450] & -1.9450 & 0.5134 & -225.6389\\
			\hline
			3 & 0.0000 & 1508 & 18.80\% & [-0.9350, 0.9350] & 0.0000 & 0.4700 & -208.5507\\
			\hline
			4 & 1.8275 & 1417 & 20.10\% & [0.9450, 2.9450] & 1.9450 & 0.5134 & -225.6389\\
			\hline
			5 & 3.8225 & 1260 & 20.50\% & [2.9550, 4.9950] & 3.9750 & 0.5335 & -241.9120\\
			\hline
		\end{tabular}
	\end{center}
	\begin{center} 
		\includegraphics[width=.8\textwidth]{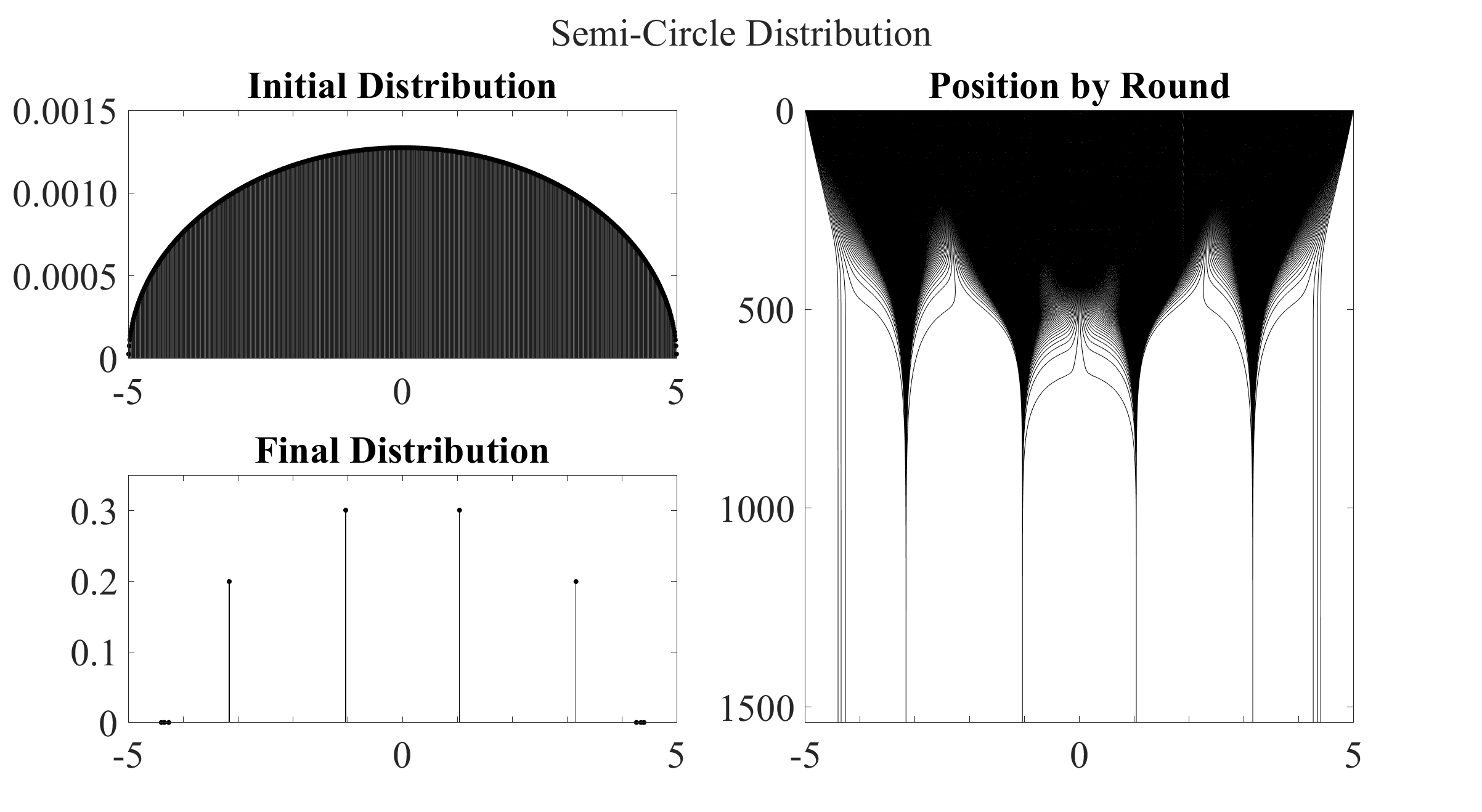}
		\begin{tabular}{|M{.06\textwidth}|M{.10\textwidth}|M{.12\textwidth}|M{.12\textwidth}|M{.18\textwidth}|M{.09\textwidth}|M{.11\textwidth}|M{.11\textwidth}|}
			\hline
			Group & Coalescing Point & Algorithm Iterations to Coalesce & Total \% of Population & Range of Initial Positions & Average Initial Position & Average Total Movement & Average Total Cost \\
			\hline
			1 & -3.1619 & 1535 & 19.94\% & [-4.9650, -2.4650] & -3.4983 & 0.6729 & -227.6249\\
			\hline
			2 & -1.0379 & 1169 & 30.03\% & [-2.4550,-0.0050] & -1.2026 & 0.6533 & -321.4321\\
			\hline
			3 & 1.0379 & 1169 & 30.03\% & [0.0050, 2.4550] & 1.2026 & 0.6533 & -321.4321\\
			\hline
			4 & 3.1619 & 1535 & 19.94\% & [2.4650, 4.9650] & 3.4983 & 0.6729 & -227.6249\\
			\hline
		\end{tabular}
	\end{center}
	\begin{center} 
		\includegraphics[width=.8\textwidth]{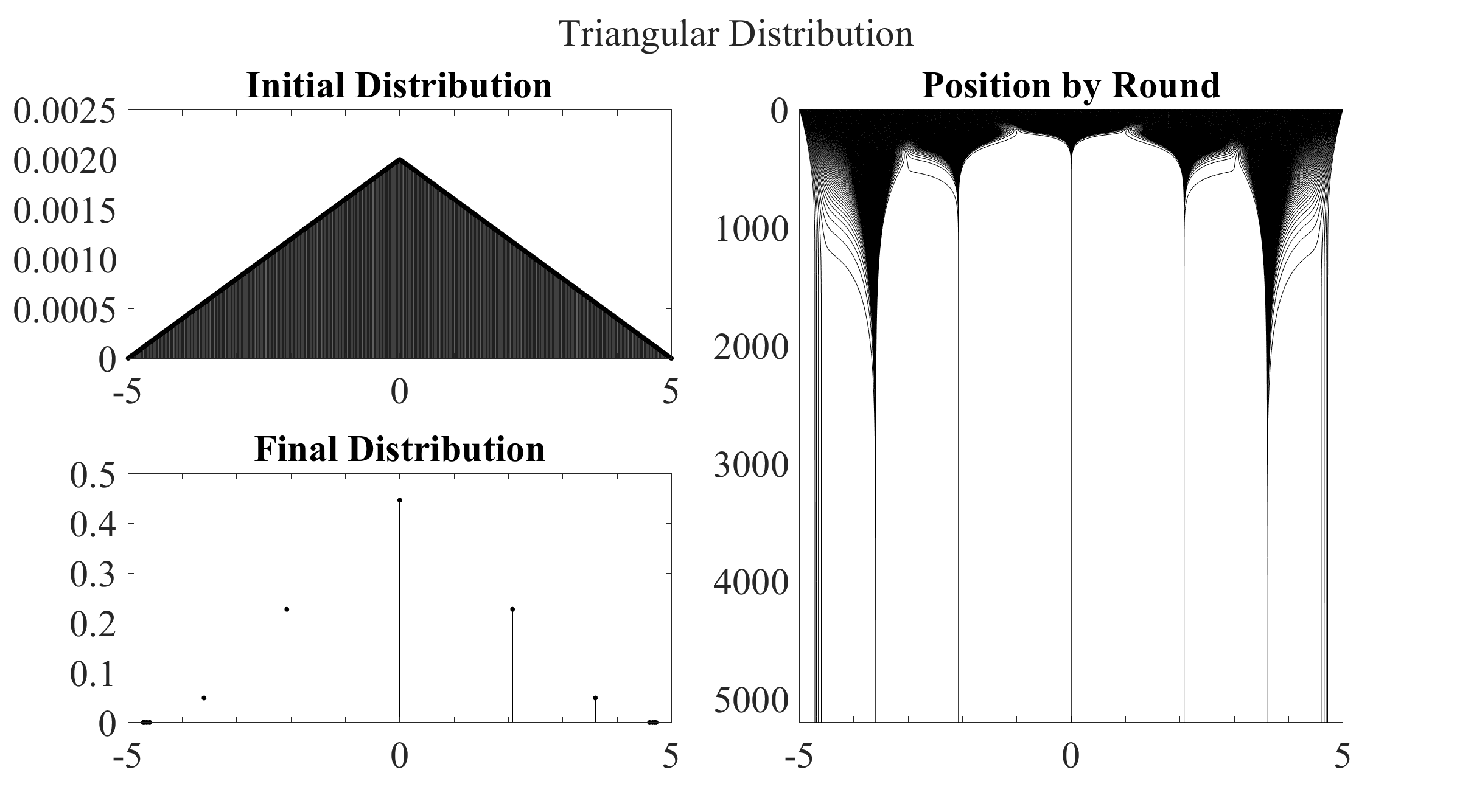}
		\begin{tabular}{|M{.06\textwidth}|M{.10\textwidth}|M{.12\textwidth}|M{.12\textwidth}|M{.18\textwidth}|M{.09\textwidth}|M{.11\textwidth}|M{.11\textwidth}|}
			\hline
			Group & Coalescing Point & Algorithm Iterations to Coalesce & Total \% of Population & Range of Initial Positions & Average Initial Position & Average Total Movement & Average Total Cost \\
			\hline
			1 & -3.5989 & 5131 & 4.93\% & [-4.9550, -3.4350] & -3.9527 & 0.6052 & -206.4532\\
			\hline
			2 & -2.0759 & 1355 & 22.75\% & [-3.4250, -1.2850] & -2.2094 & 0.6614 & -943.5018\\
			\hline
			3 & 0.0000 & 639 & 44.65\% & [-1.2750, 1.2750] & 0.0000 & 0.6087 & -1869.5074\\
			\hline
			4 & 2.0759 & 1355 & 22.75\% & [1.2850, 3.4250] & 2.2094 & 0.6614 & -943.5018\\
			\hline
			5 & 3.5989 & 5131 & 4.93\% & [3.4350, 4.9550] & 3.9527 & 0.6052 & -206.4532\\
			\hline
		\end{tabular}
	\end{center}	
	\begin{center} 
		\includegraphics[width=.8\textwidth]{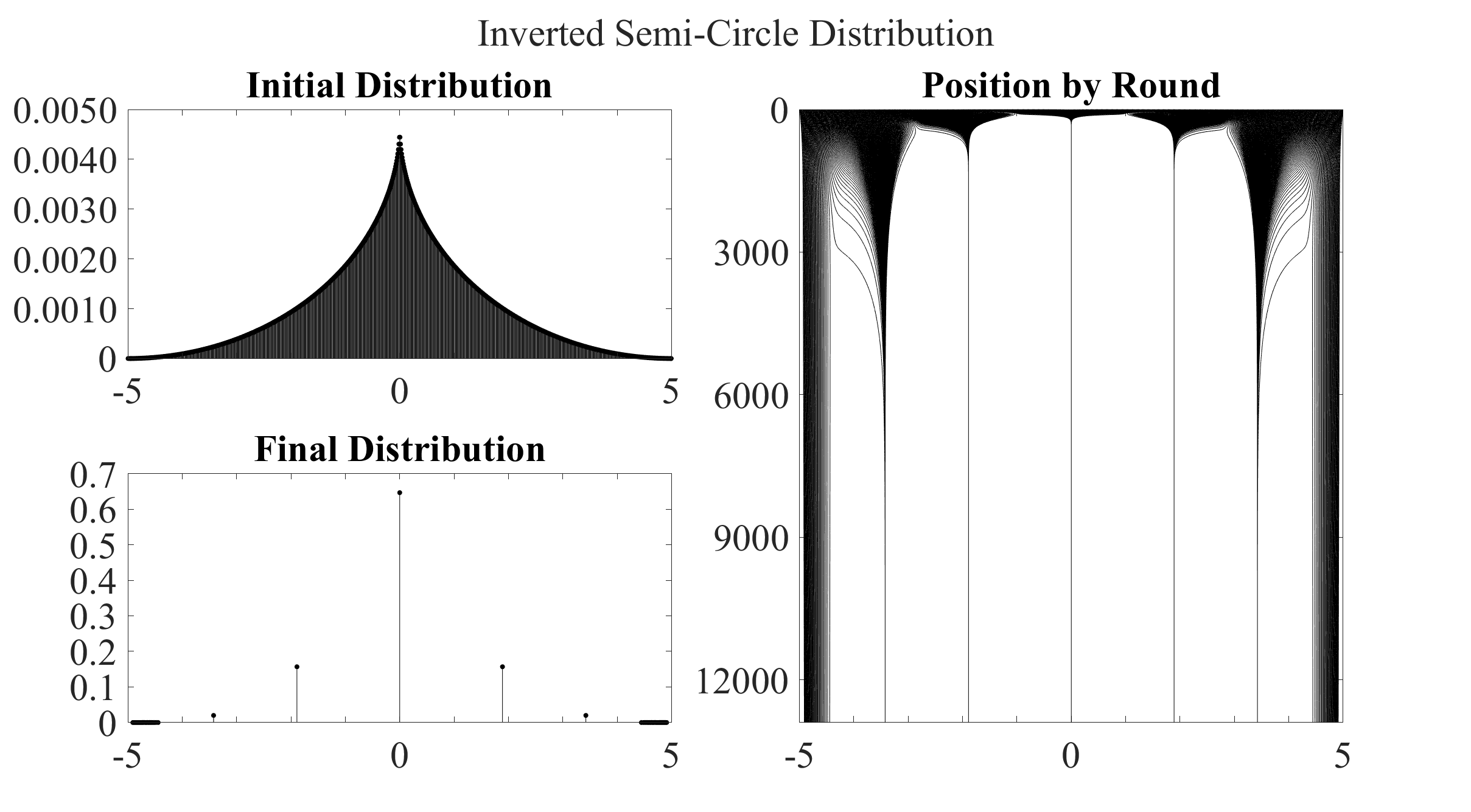}
		\begin{tabular}{|M{.06\textwidth}|M{.10\textwidth}|M{.12\textwidth}|M{.12\textwidth}|M{.18\textwidth}|M{.09\textwidth}|M{.11\textwidth}|M{.11\textwidth}|}
			\hline
			Group & Coalescing Point & Algorithm Iterations to Coalesce & Total \% of Population & Range of Initial Positions & Average Initial Position & Average Total Movement & Average Total Cost \\
			\hline
			1 & -3.4234 & 12806 & 2.00\% & [-4.6650, -3.1550] & -3.6028 & 0.4675 & -209.4471\\
			\hline
			2 & -1.8900 & 1604 & 15.68\% & [-3.1450, -1.2850] & -1.9927 & 0.5479 & -1653.3058\\
			\hline
			3 & 0.0000 & 406 & 64.62\% & [-1.2750, 1.2750] & 0.0000 & 0.5369 & -6841.9609\\
			\hline
			4 & 1.8900 & 1604 & 15.68\% & [1.2850, 3.1450] & 1.9927 & 0.5479 & -1653.3058\\
			\hline
			5 & 3.4234 & 12806 & 2.00\% & [3.1550, 4.6650] & 3.6028 & 0.4675 & -209.4471\\
			\hline
			
		\end{tabular}
	\end{center}
	
	In the graphs above, the initial and final distribution are graphed by plotting points $(x,y)$ where $x$ is the location and $y$ is the total proportion of the population holding position $x$.
	The ``Position by Round'' graph displays only the distribution of $x$-values at each round, starting from the top (the population density itself is not directly visible in this graph).
	In all four cases, the iterative process of each agent minimizing their own total cost function divided the majority of the population into clusters, but only the Uniform Distribution simulation saw every agent eventually belong to a cluster. The other three distributions ended the simulation with a small number of isolated agents at the extreme ends of each distribution in the lowest density regions. These isolated agents represented $0.0431\%$ of the Semi-Circle Distribution, $0.0065\%$ of the Triangular Distribution, and $0.0224\%$ of the Inverted Semi-Circle Distribution. Isolated agents in these low density regions begin the iterative process with the same ingathering behavior as their interior neighbors and move towards the distribution's mean for many iterations. However as their inlying neighbors start moving inwards faster towards a nearby accumulating mass of other agents, the isolated agents get left behind, and the innermost isolated agents begin moving outwards towards the still ingathering, most extreme isolated agents. These low density regions of isolated agents may eventually coalesce into groups, but the meager benefit of gathering in such low density regions results in incredibly slow movement from residing agents. Consequently, waiting for these regions to coalesce would take more iterations than can be reasonably observed.
	
	Conversely, agents in the highest density regions tend to coalesce the fastest. Group 3 in the Inverted Semi-Circle Distribution simulation comprised $64.62\%$ of the population and coalesced in a speedy 406 iterations. Altogether, the Uniform Distribution had completely coalesced into clusters after 1508 iterations with its tails coalescing first after 1260 iterations. This observed behavior of the highest density regions coalescing the fastest and lowest density regions moving too slow to coalesce completely aligns with the structure of the individual's total cost function. High density regions present a larger reward for gathering, so individual agents can justify more movement in a single iteration. The opposite is true for low density regions. As a result, the region centered about the mean contained the largest group at the end of the simulation in almost every case. The Uniform Distribution simulation notably deviates from this trend with the outermost groups containing $1.7\%$ more of the population than the central group at the mean. This is a direct result of the distribution's fat tails. Agents in the tail ends of the distribution are markedly incentivized to get away from the empty regions next to them and generally move inwards quickly until enough of their inlying neighbors find their accumulating mass sufficiently attractive to meet them to form a single mass. In the Uniform Distribution simulation, this quick piling up of the already dense tails produces a large, attractive mass that outsizes the more inward accumulating masses. 
	
	Moreover, agent mobility varied across distribution and final group assignment. The average total movement across all iterations to coalescence for a single agent was 0.5135 in the Uniform Distribution simulation, 0.6609 in the Semi-Circle Distribution simulation, 0.6323 in the Triangular Distribution simulation, and 0.5374 in the Inverted Semi-Circle Distribution simulation. Due to their fatter tails, the Uniform and Semi-Circle Distribution simulations had decreasing average total movement from the outermost groups to the innermost group(s), which is a result of the outermost agents trying to rapidly get away from the empty region next to them. The least mobile groups in the Triangular and Inverted Semi-Circle Distribution simulations with thinner tails were the outermost groups, but similar to the Uniform and Semi-Circle Distribution simulations, the middling groups had a higher average total movement then the innermost group. Low population density in the outermost regions slows down the groups and negatively affects the groups' average total movement. Meanwhile, the middling regions may not have completely empty regions next to them, the outermost regions are comparatively empty to the innermost regions, producing the same effect on the middling and innermost groups as the Uniform and Semi-Circle Distribution simulations. Within groups, agents that happened to begin the simulations near the coalescing point for their group did not have to move very far to converge on their final position, while agents that happened to begin further away from the coalescing point for their group had to move significantly to converge on their final position. 
	
	Population density had the biggest impact on average total cost. Groups with the highest population density had the lowest average total cost accumulated across all rounds to coalescence. Like total movement for individual agents, the agents that happened to begin the simulations near the coalescing point for their group had lower total cost than the agents that happened to begin the simulations farther away from the coalescing point for their group.
	
	Finally, in every case we examined, the final distribution was narrower than the initial distribution, but not every agent consistently moved towards the mean. $35.20\%$ of agents in the Uniform Distribution, $41.04\%$ of agents in the Semi-Circle Distribution, $23.26\%$ of agents in the Triangular Distribution, and $16.96\%$ of agents in the Inverted Semi-Circle Distribution ended up farther away from the mean of the distribution in their final position than in their initial position. However, the final coalescing point of every group is closer to the mean of each distribution than their average initial position. The narrowing of the population in addition to the difference between the final coalescing point and average initial position of each group indicates an overall mean-convergent behavior, but the existence of separate groups that do not all converge to the mean as well as the not insignificant portion of each population distribution moving away from the mean implies something interesting is going on. If the population in all four simulations has an overall ingathering behavior, why do distinct groups form? The observed effect of the iterative process of each agent minimizing their own total cost function is not wholesale ``polarization'' because the vast majority of agents end the simulation in more moderate positions than they started in. Rather in all four simulations, we can see the flight of the relative-moderate who is attracted to the mass of agents with more extreme positions that are willing to compromise and exhibit mean-converging behavior.

	Altogether, a population distribution's density and concavity affects the formation of groups, especially at the tails.

	\section{Conclusion} \label{sec:conclusion}
	
	In this paper we have studied a simple mean field game in which players have an incentive to congregate, and we have analyzed the resulting dynamical system formed by iterating the game.
	We have identified the fixed points and rigorously proved the asymptotic stability for this dynamical system.
	Our numerical simulations both demonstrate the validity of these results and also provide thought-provoking details about the dynamics.
	We now conclude this paper with some open questions for both theory and applications.
	
	\subsection{Equilibrium for a long time horizon}
	
	One natural question is whether the stable fixed points of the dynamical system  \eqref{eq:Etilde} correspond (at least approximately) to Nash equilibrium points for the game with a very long time horizon.
	A heuristic argument in favor of this conjecture is as follows.
	We take $G(y,\mu) = \int \varphi(y-z)\dif \mu(z)$ and consider its discrete analog $g(y,{\bf z}) = \frac{1}{n}\sum_{k=1}^n \varphi(y-z_k)$.
	Recall that the cost to each player moving from position $x$ to $y$ is
	\begin{equation*}
		\frac{\abs{x-y}^2}{2t} + \frac{1}{n}\sum_{k=1}^n \varphi(y-z_k).
	\end{equation*}
	Taking $t \to \infty$, we see that $x$ is almost irrelevant.
	The optimal strategy is essentially to find a minimizer of $y \mapsto \sum_{k=1}^n \varphi(y-z_k)$.
	If ${\bf z} = (z_1,\ldots,z_n)$ is to represent an equilibrium measure, then each $z_j$ must in fact be a minimizer of this function.
	Assume $-\varphi$ is a bump function as in \eqref{eq:bump}.
	Glancing at the graph of $y \mapsto \sum_{k=1}^n \varphi(y-z_k)$ suggests ${\bf z}$ is an equilibrium if and only if the points $z_1,\ldots,z_n$ form equal sized clusters $z_1 = \cdots = z_{k} < z_{k + 1} = \cdots = z_{2k} < \cdots < z_{mk}$ such that $z_{(j+1)k} - z_{jk}$ is sufficiently large for each $j$.
	Hence ${\bf z}$ should also be a stable fixed point of the dynamical system \eqref{eq:Etilde} according to Theorem \ref{thm:fixpoint}.
	Since we cannot expect the equilibrium to be unique for an arbitrary time horizon, this infinite time limit case might provide some meaningful way of selecting among multiple equilibria for large time horizon games.
	We intend to investigate this in future research.
	
	\subsection{Opinion dynamics}
	
	One particularly compelling application for our model is where the population measure represents the distribution of people on a spectrum of opinions, e.g.~a single-issue political spectrum.
	The iterated game could model a process in which people change their positions incrementally, trying to find communities with valuable social identities of like-minded peers, or perhaps trying to build coalitions to win political capital.
	On an optimistic note, we observe that in our simulations the entire population becomes more moderate as a whole, and most individuals elect to compromise to join less extreme communities.
	However, the population's separation into distinct, consolidated groups gives dissidents platforms with greater leverage than they began with, and the most extreme dissidents are left behind entirely with their voices completely drowned out by more popular groups.
	Our model shows that the individual proclivity to group can reshape a population at large, but the effects of the population's new shape on society are harder to determine. 
	It would be interesting to use game theoretic models to further investigate the impact of grouping on political or other social systems. 
	
	\section*{Declarations}
	
	\noindent \textbf{Ethical Approval.} Not applicable.

	\noindent {\bf Competing interests.} The authors declare that they do not have any conflicts of interests.

	\noindent {\bf Authors' contributions.} The authors have each contributed equally to the present manuscript.

	\noindent {\bf Funding.} The authors gratefully acknowledge the support of the National Science Foundation through research grant DMS-2045027.
	
	\noindent {\bf Data Availability Statement.} Data sharing not applicable to this article as no external datasets were generated or analyzed during the current study.
	
	\bibliographystyle{alpha}
	\bibliography{../../mybib/mybib}
	
\end{document}